\documentclass[12pt,letterpaper]{article}
\usepackage{setspace}
\usepackage[margin=1.2in]{geometry}

\usepackage{epsfig, amsthm, amsmath, amssymb, latexsym, xpatch, xcolor, cancel}
\usepackage[titletoc]{appendix}
\usepackage{float} 
\usepackage[margin=0.5cm, font={small,it}]{caption}
\usepackage{wrapfig}
\usepackage[normalem]{ulem} 

\usepackage[mathscr]{euscript} 

\usepackage{graphicx}

\theoremstyle{plain}
\newtheorem{theorem}{Theorem}
\newtheorem*{theorem*}{Theorem}
\newtheorem{proposition}{Proposition}
\newtheorem{lemma}{Lemma}
\newtheorem*{lemma*}{Lemma}

\theoremstyle{definition}

\newtheorem*{definition*}{Definition}

\newtheorem*{example*}{Example}
\newtheorem*{remark*}{Remark}

\newcommand{\To}{\Rightarrow}

\newcommand{\op}{\begin{itemize}}
\newcommand{\ed}{\end{itemize}}

\newcommand{\opp}{\begin{quote}}
\newcommand{\edd}{\end{quote}}

\newcommand{\ope}{\begin{enumerate}}
\newcommand{\ede}{\end{enumerate}}

\newcommand{\im}{\item}

\newcommand{\PP}{\mathbb{P}}
\newcommand{\phii}{\varphi}

\title{Meager Success: A Theory of the Unlearnable for Hypothesis Testing}

\author{Hanti Lin \\[0.5em] University of California, Davis \\ika@ucdavis.edu}

\date{}

\begin{document}

\maketitle

\begin{abstract} \noindent When the standard of pointwise consistency for statistical inference---convergence to the truth in every possible state of the world---is provably unachievable, the usual responses are to change the inferential target or to strengthen background assumptions. This paper pursues a third: hold the inference problem fixed and identify the highest standard that remains achievable. I define a hierarchy of standards weaker than pointwise consistency, cast in topological terms, requiring convergence to the truth not everywhere but on a ``large'' set of probability measures. The main result is an impossibility theorem: for finite-precision tests, converging to the truth densely within each hypothesis already forces inconsistency on a comeager---``topologically almost all''---set of measures, whenever the two hypotheses are dense in their union. Distribution-free testing of conditional independence is one such case. Two further theorems characterize, in purely topological terms about the structures of the tested hypotheses, exactly when each of the weaker standards under study is achievable. 
\end{abstract} 


\section{Introduction}

How should we respond when the standard we would like an inference procedure to satisfy turns out to be unachievable? To fix ideas, begin with the standard most often taken for granted: {\em pointwise consistency}, that is, convergence (in probability) to the true learning target in every state of the world compatible with the background assumptions. This might appear to be a minimum qualification on any inference procedure. When this standard is provably too high to be achievable, there are two usual responses. Way 1: change the learning target, such as estimating, not the average treatment effect, but the complier average treatment effect (Angrist, Imbens \& Rubin 1996). Way 2: add background assumptions, such as the smoothness condition about the underlying distribution or the causal faithfulness assumption in causal discovery, i.e., causal graphical structure learning (Spirtes, Glymour \& Scheines 2000). 

There is a third way, although it is less familiar. In causal discovery, Lin and Zhang (2020) propose that we do not have to always pile on assumptions. Instead, without adding background assumptions, we can explore lower standards, asking whether some standard low enough to be achievable is still high enough to justify our inference procedures. In this spirit, Lin and Zhang (2020) explore the standard of convergence to the true learning target in, not all, but {\em almost all}, states of the world compatible with the background assumptions---``almost all'' in a topological sense. Once such a standard is provably achievable, the next step is to ask whether we can raise the bar without sacrificing achievability. The idea, put roughly, is this: {\em explore what standards can be achieved, and strive for the highest achievable one.} This is the third way I have in mind. 

Although less familiar in this form, the underlying idea is in fact {\em old}---in a sense, textbook standard. What Lin and Zhang (2020) do not note, and I hasten to add, is that the spirit of exploring what can be achieved and striving for the highest achievable one has a venerable precedent. Soon after Neyman and Pearson's (1933) work on the standard of {\em uniformly most powerful testing}, they recognized that this standard is too high to be achievable in two-sided problems. What they did was not to change the learning target or to pile on assumptions. Indeed, in Neyman and Pearson's (1936) follow-up, they explored a lower standard: uniformly most powerful {\em unbiased} testing. 

To be sure, the principle of achieving the highest achievable needs a qualification: when the highest achievable is shown to be too low a standard, it is time to get back to Way 1 or Way 2. That is, there should be a {\em minimum qualification} for justified inference procedures. And the principle should be: achieving the highest achievable, unless that is too low to be worth striving for---lower than the minimum qualification. 

This raises a philosophical question: what should the minimum qualification be? Lin and Zhang (2020) can be understood to propose that the minimum qualification should be lower than pointwise consistency. And I can imagine that some would object, defending the view that the minimum qualification is pointwise consistency. Perhaps some would even argue that the minimum qualification is higher---being, say, uniform consistency.

This is not the place to engage with this philosophical debate. But some mathematical work can be done to facilitate the debate. Evaluative standards lower than pointwise consistency can be explored and rigorously defined, and the question of when they are achievable can be mathematically studied, hopefully with some theorems. Examples of inference problems in which those standards are achievable or unachievable can also be studied mathematically. It is only once we have such theorems and examples in hand that we can do the philosophy in needed---the philosophy of statistics that the debate calls for: if we set the minimum qualification here or there, what must be sacrificed---which targets changed, or which assumptions added, and in which inference problems? 

Here is an example in point: the problem of testing conditional independence given a real-valued variable in a fully distribution-free setting is very hard, so much so that, in a sense, any reasonable test (or sequence of tests) is doomed to be inconsistent relative to ``almost all'' probability measures---in the sense of ``almost all'' used in Banach's (1931) celebrated theorem that ``almost all'' continuous functions are nowhere differentiable, that is, ``almost all'' in the sense of being comeager in topology. This is one application of a new impossibility result (Theorem 1, in Section 3). 

I also define various standards---lower than pointwise consistency---for assessing tests, and prove necessary and sufficient conditions for when each is achievable (Theorems 2 and 3, in Sections 4 and 6; Section 5 records a structural characterization of the key condition). Those conditions are topological, regarding the topological structures of the tested hypotheses (under the weak topology). This can be understood as a companion to the work of Genin \& Kelly (2017) and Boeken, Skapinakis, Genin \& Mooij (2026) on a topological characterization of when pointwise consistency is achievable, and the work of Kelly (1996) on a topological characterization for the deterministic counterpart of pointwise consistency (in formal learning theory).

\section{Definitions}

Let ${\cal X}$ be a space of possible data points, assumed to be Polish (i.e., separable and completely metrizable). {\em Hypotheses} are identified with sets of probability measures on $\mathscr{B}({\cal X})$, the Borel $\sigma$-algebra over ${\cal X}$. A {\em hypothesis testing problem} consists of two disjoint hypotheses $H_0$ and $H_1$.  

Given a hypothesis testing problem $H_0$ versus $H_1$, a {\em test} $\phii$ is a sequence of measurable functions $\phii_1, \phii_2, \ldots$ such that $\phii_n: {\cal X}^n \to \{0, 1, \texttt{?}\}$ sends each data sequence of length $n$ to one of the two competing hypotheses (as its inferential output). Let $\{\phii_n = i\}$ denote the set $\{x \in {\cal X}^n: \phii_n(x) = i\}$. A test $\phii$ is said to be {\em consistent} at $\PP  \in H_i$ iff $\PP^{\;\!n}\{\phii_n = i\} \to 1$ as $n$ increases indefinitely---that is, the probability of identifying the true hypothesis converges to 1. Note that $\PP^{\;\!n}$ denotes the $n$-fold product measure of $\PP$; so, IID data are assumed.

Let the space of all probability measures on $\mathscr{B}({\cal X})$ be equipped with the weak topology, and let each subset of this space be equipped with the subset topology. In a topological space, define topological sizes as follows
	\op
	\im A set $S$  is called {\em nowhere dense} (or {\em very meager}) iff every open set that overlaps $S$ contains a nonempty open subset disjoint from $S$---that is, $S$ is ``full of holes''.
	\im A set is called {\em meager} iff it is a countable union of nowhere dense sets.
	\im A set is called {\em nonmeager} iff it is not meager.
	\im A set is called {\em comeager} iff its complement (in the topological space in question) is meager.
	\ed 

Let $\phii$ be a test for a hypothesis testing problem $H_0$ versus $H_1$. When $\phii$ has a certain property at each probability measure in set $S$ (as a subset of a topological space), say that $\phii$ has that property {\em on} $S$. Define standards for assessing $\phii$ as follows, from high to low:
	\op
	\im $\phii$ is called {\em everywhere consistent}---but more commonly called {\em pointwise consistent}---iff $\phii$ is consistent on $H_0 \cup H_1$.
		
	\im $\phii$ is called {\em comeagerly consistent} iff $\phii$ is consistent on some comeager subset of the union $H_0 \cup H_1$ (as a topological space).
	
	\im $\phii$ is called {\em nonmeagerly consistent} iff $\phii$ is consistent on some nonmeager subset of the union $H_0 \cup H_1$.
	\ed 
Violation of the lowest of the above standards, nonmeager consistency, is disastrous: it means being inconsistent on a comeager subset of the union $H_0 \cup H_1$---it is {\em comeager inconsistency}. Last, let me introduce what should go into any reasonable minimum qualification:
	\op 
	\im $\phii$ is called {\em densely consistent in each hypothesis} iff, for each $i \in \{0, 1\}$, $\phii$ is consistent on some dense subset of $H_i$.
	\ed  
Following Boeken, Skapinakis, Genin \& Mooij (2026), say that a test $\phii$ for a problem $H_0$ versus $H_1$ has the {\em finite-precision} property iff, for any sample size $n$ and any hypothesis $H_i$ with $i \in \{0, 1\}$, its acceptance region $\{\phii_n = i\}$ is open.

\section{The Main Impossibility Result}

\begin{theorem} In any hypothesis testing problem $H_0$ versus $H_1$ such that each hypothesis $H_i$ is dense in the union $H_0 \cup H_1$ $($with respect to the weak topology$)$, no finite-precision test achieves the following two standards simultaneously: 
	\op 
	\im[1.] dense consistency in each hypothesis,
	\im[2.] nonmeager consistency.
	\ed 
\end{theorem}

\begin{remark*}
Violation of 2, nonmeager consistency, means comeager inconsistency. But reading comeager inconsistency as large-scale failure presumes that ``comeager'' carries its intended sense of topological largeness, which holds precisely when the ambient space $H_0 \cup H_1$ is a {\em Baire} space: there every comeager set is dense by definition, hence nonempty. To be sure, Theorem~1 does not require that $H_0 \cup H_1$ is Baire. What Baireness supplies is not the conclusion but its interpretation. The distinction is not idle: in a non-Baire space the notion can degenerate, for if $H_0 \cup H_1$ is not a Baire space, then it might be that every subset is meager, hence every subset is comeager, so that even the empty set is comeager.\footnote{The rationals under the subspace topology inherited from $\mathbb{R}$ are the standard witness: being countable and without isolated points, $\mathbb{Q}$ is meager in itself, and there ``comeager'' conveys nothing.} The impossibility content of Theorem~1 is therefore secured, for a given testing problem, by checking that its space $H_0 \cup H_1$ is Baire---as it is in each application below, the union being a Polish space and hence Baire. So, in such a problem, if the space $H_0 \cup H_1$ of the considered probability measures is Baire, then any design of a finite-precision test that ensures at least dense consistency in each hypothesis is doomed to incur comeager inconsistency---having a large-scale inconsistency in the topological sense.
\end{remark*}

\begin{example*}
One such problem is the problem of testing the conditional independence between two random variables $X$ and $Y$ given a real-valued random variable $Z$, against its negation, in the distribution-free setting. The denseness condition is satisfied, thanks to Boeken, Skapinakis, Genin \& Mooij (2026, Lemmas 4 and 5) and Boeken, Forr\'{e} \& Mooij (2026, Corollary 1). Moreover, the distribution-free setting satisfies the condition of being a Baire space (the same holds when we add the assumption that requires probability measures to be atomless). So, by the impossibility result presented above, i.e., Theorem 1, this hypothesis testing problem is indeed very hard.
\end{example*}
 
 

The rest of this section is devoted to proving Theorem 1. A test $\phii$ is said to {\em vacillate at least $k$ times} at a probability measure $\PP$ iff there exists an increasing sequence of sample sizes $n_0 < n_1 < n_2 < \ldots < n_k$ such that 
	\op 
	\im $\PP^{\;\! n_i} (\phii_{n_i} = 0) > 0.9$ for each even-numbered $n_i$ in the sequence,
	\im $\PP^{\;\! n_j} (\phii_{n_j} = 1) > 0.9$ for each odd-numbered $n_j$ in the sequence.
	\ed 
We may replace $0.9$ by any threshold above $0.5$.

\begin{lemma}[{The Vacillation Lemma}] Consider any hypothesis testing problem $H_0$ versus $H_1$ such that each hypothesis $H_i$ is dense in the union $H_0 \cup H_1$. Let $\phii$ be a finite-precision test that is densely consistent in each hypothesis. Then, for each positive integer $k$, every nonempty open subset $O \subseteq H_0 \cup H_1$ contains a nonempty open $O'$ subset on which $\phii$ vacillates at least $k$ times.
\end{lemma}

\begin{proof}
Fix $k$ and let $O \subseteq H_0 \cup H_1$ be a nonempty open subset. We construct, by recursion on $i = 0, 1, \ldots, k$, a decreasing chain of nonempty open sets
\[
	O = O_{-1} \supseteq O_0 \supseteq O_1 \supseteq \cdots \supseteq O_k
\]
together with an increasing sequence of sample sizes $n_0 < n_1 < \cdots < n_k$ such that, for each $i$,
\[
	\PP^{\;\! n_i}(\phii_{n_i} = c_i) > 0.9 \quad \text{for every } \PP \in O_i,
\]
where $c_i := i \bmod 2$ alternates $0, 1, 0, 1, \ldots$ Taking $O' := O_k$ then yields a nonempty open subset of $O$ on which $\phii$ vacillates at least $k$ times, since every $\PP \in O_k \subseteq O_i$ satisfies the required inequality at $n_i$ for each $i \le k$.

For the recursion, suppose $i \ge 0$ and the nonempty open set $O_{i-1}$ has been constructed (with $O_{-1} = O$), and suppose sample sizes $n_0 < \cdots < n_{i-1}$ have been chosen (vacuously if $i = 0$). Write $c := c_i \in \{0,1\}$ for the hypothesis targeted at this stage.

\smallskip
\noindent\textit{Step 1: pass to a point of $H_c$.} Since $H_c$ is dense in $H_0 \cup H_1$ and $O_{i-1}$ is a nonempty open subset of $H_0 \cup H_1$, the intersection $O_{i-1} \cap H_c$ is nonempty. Because $\phii$ is densely consistent in each hypothesis, it is consistent on some dense subset $D_c \subseteq H_c$; and a dense subset of $H_c$ meets every nonempty relatively open subset of $H_c$, in particular the nonempty relatively open set $O_{i-1} \cap H_c$. Hence there exists a distribution
\[
	\PP^\ast \in O_{i-1} \cap H_c
\]
at which $\phii$ is consistent, i.e. $(\PP^\ast)^{\,n}(\phii_n = c) \to 1$ as $n \to \infty$.

\smallskip
\noindent\textit{Step 2: choose the sample size.} By the convergence in Step 1, choose a sample size $n_i > n_{i-1}$ (with $n_i > 0$ when $i = 0$) large enough that
\[
	(\PP^\ast)^{\,n_i}(\phii_{n_i} = c) > 0.9 .
\]
Such an $n_i$ exists because the limit is $1 > 0.9$; and we may always take it strictly larger than $n_{i-1}$, since the convergence guarantees the inequality holds for all sufficiently large indices.

\smallskip
\noindent\textit{Step 3: open out from the point.} By the finite-precision property, the acceptance region $A := \{\phii_{n_i} = c\} \subseteq {\cal X}^{n_i}$ is open. Consider the set
\[
	U_{n_i} := \bigl\{ \PP \in H_0 \cup H_1 : \PP^{\,n_i}(\phii_{n_i} = c) > 0.9 \bigr\}.
\]
By Step 2 we have $\PP^\ast \in U_{n_i}$, and by Step 1 also $\PP^\ast \in O_{i-1}$. We claim $U_{n_i}$ is open in $H_0 \cup H_1$.
 
Suppose, for reductio, that $U_{n_i}$ is not open. Then some $\PP \in U_{n_i}$ fails to be an interior point: every weak-topology neighbourhood of $\PP$ contains a point of $H_0 \cup H_1$ lying outside $U_{n_i}$. Since ${\cal X}$ is Polish, the space of Borel probability measures on ${\cal X}$ under the weak topology is metrizable, hence first countable; fixing a countable neighbourhood base at $\PP$ and shrinking, we extract a sequence $(\PP_m)_{m \ge 1}$ in $H_0 \cup H_1$ with
\[
	\PP_m \to \PP \ \text{weakly}, \qquad \PP_m \notin U_{n_i} \ \text{for every } m,
\]
so that $\PP_m^{\,n_i}(\phii_{n_i} = c) \le 0.9$ for all $m$. Weak convergence $\PP_m \to \PP$ implies $\PP_m^{\,n_i} \to \PP^{\,n_i}$ weakly on the product space ${\cal X}^{n_i}$. Applying the Portmanteau theorem in its lower-semicontinuity form to the open set $A$,
\[
	\PP^{\,n_i}(A) \;\le\; \liminf_{m \to \infty} \PP_m^{\,n_i}(A) \;\le\; 0.9,
\]
which contradicts $\PP \in U_{n_i}$, i.e. $\PP^{\,n_i}(A) > 0.9$. Hence $U_{n_i}$ is open.
 
Since $\PP^\ast \in O_{i-1} \cap U_{n_i}$, the set
\[
	O_i := O_{i-1} \cap U_{n_i}
\]
is a nonempty open subset of $O_{i-1}$ (it contains $\PP^\ast$). By construction every $\PP \in O_i$ satisfies $\PP^{\,n_i}(\phii_{n_i} = c_i) > 0.9$, completing the recursion step.

\smallskip
After $k+1$ steps we obtain $O' := O_k \subseteq O$, nonempty and open, and sample sizes $n_0 < n_1 < \cdots < n_k$ with $\PP^{\,n_i}(\phii_{n_i} = i \bmod 2) > 0.9$ for every $\PP \in O'$ and every $i \le k$. Setting the even-indexed sample sizes to target $H_0$ and the odd-indexed ones to target $H_1$, this is exactly the statement that $\phii$ vacillates at least $k$ times on $O'$.
\end{proof}

\medskip
\noindent We now prove the theorem. Throughout, topological notions (meager, comeager, interior) are understood relative to the union $H_0 \cup H_1$ with its subspace topology, as stipulated in the definitions.

\begin{proof}[Proof of Theorem 1]
Let $\phii$ be a finite-precision test that is densely consistent in each hypothesis; we show $\phii$ is not nonmeagerly consistent. For each positive integer $k$, let
\[
	V_k := \{ \PP \in H_0 \cup H_1 : \phii \text{ vacillates at least } k \text{ times at } \PP \}
\]
be the set of distributions at which $\phii$ vacillates at least $k$ times. Let $I$ denote the set of distributions in $H_0 \cup H_1$ at which $\phii$ is \emph{not} consistent. The argument has two parts: first, $\bigcap_{k \ge 1} V_k$ is comeager; second, $\bigcap_{k \ge 1} V_k \subseteq I$. Together these make $I$ comeager, so that every set on which $\phii$ is consistent---being a subset of $I^c$---is meager, which is exactly the failure of nonmeager consistency. The two parts are established as follows.

\textit{Part 1: $\bigcap_{k\ge 1} V_k$ is comeager.} Fix $k$. The Vacillation Lemma says that every nonempty open set contains a nonempty open subset of $V_k$; equivalently, $V_k^c$ is nowhere dense. Hence $\bigl(\bigcap_{k} V_k\bigr)^c = \bigcup_k V_k^c$ is a countable union of nowhere dense sets---meager---so $\bigcap_{k} V_k$ is comeager.

\textit{Part 2: $\bigcap_{k\ge 1} V_k \subseteq I$.} Let $\PP \in \bigcap_{k\ge1} V_k$, say with $\PP \in H_i$. Vacillating at least $k$ times for \emph{every} $k$ means that, at arbitrarily large sample sizes, the probability of outputting $0$ rises above $0.9$ and the probability of outputting $1$ rises above $0.9$, alternately and infinitely often. So the probability of outputting $i$ (the true hypothesis $H_i$) keeps going up and down---it exceeds $0.9$ infinitely often (when the verdict favors the truth $H_i$) yet also drops below $0.1$ infinitely often (when the verdict favors the falsehood $H_{1-i}$, since the output probabilities sum to $1$). A sequence of probabilities that oscillates in this way cannot converge to $1$, so $\PP^{\,n}(\phii_n = i) \not\to 1$: consistency fails at $\PP$, i.e. $\PP \in I$.

Combining the two parts, $I \supseteq \bigcap_{k\ge1} V_k$ is comeager, so every set on which $\phii$ is consistent is contained in the meager set $I^c$ and is itself meager. Hence $\phii$ is not nonmeagerly consistent. And since $\phii$ was an arbitrary finite-precision test densely consistent in each hypothesis, no finite-precision test achieves standards~1 and~2 simultaneously.
\end{proof}

\section{Positive Result: A Characterization Theorem}

In a problem of testing $H_0$ versus $H_1$, define the following evaluative standards:
	\op
	\im A test $\phii$ is said to achieve the standard of almost everywhere consistency in each hypothesis iff, for each $i \in \{0, 1\}$, the set of probability measures in $H_i$ at which $\phii$ is consistent can be expressed as $H_i$ minus a subset that is nowhere dense in $H_i$ (with respect to its own subspace topology).
	
	\im A test $\phii$ is said to achieve the standard of almost everywhere consistency iff the set of probability measures in $H_0 \cup H_1$ at which $\phii$ is consistent can be expressed as $H_0 \cup H_1$ minus a nowhere dense subset.

	\im A test $\phii$ is said to achieve the standard of comeager/nonmeager consistency iff the set of probability measures in $H_0 \cup H_1$ at which $\phii$ is consistent is comeager/nonmeager.
	\ed 

In a hypothesis testing problem $H_0$ versus $H_1$, the two hypotheses are called {\em intertwined} in a nonempty open set $O$ in the ambient topological space $H_0 \cup H_1$ iff both $H_0$ and $H_1$ are dense in $O$; they are called {\em nowhere intertwined} iff they are intertwined in no nonempty open set in $H_0 \cup H_1$. 

\begin{theorem}
For any hypothesis testing problem $H_0$ versus $H_1$ whose ambient space $H_0 \cup H_1$ is a Baire space $($which holds, for instance, whenever $H_0 \cup H_1$ is Polish$)$, the following conditions are equivalent:
\ope  
\im The two hypotheses are nowhere intertwined.

\im There exists a finite-precision test procedure that achieves the $($higher$)$ standard of almost everywhere consistency in each hypothesis.

\im There exists a finite-precision test procedure that achieves the $($intermediate$)$ standard of $($i$)$ dense consistency in each hypothesis and $($ii$)$ almost everywhere consistency.
\im There exists a finite-precision test procedure that achieves the $($lower$)$ standard of $($i$)$ dense consistency in each hypothesis and $($ii$)$ comeager consistency.
\ede  
\end{theorem}

Although the standards in conditions 2, 3, and 4 differ---ranging from higher to lower---they turn out to be achievable under exactly the same condition (given the Baire assumption). This suggests a methodological moral: when a lower standard mentioned in 3 or 4 is achievable, one need not settle for it---one may aim for the higher one mentioned in 2.

The proof proceeds along the cycle $1 \To 2 \To 3 \To 4 \To 1$. Of these, $4 \To 1$ is the only part that relies on the Baire assumption, and it is essentially a corollary of the main impossibility result, Theorem 1, applied not to the original problem $H_0$ versus $H_1$ but to an ``open subproblem'' of it. Also, $1 \To 2$ is an existence claim, whose proof relies on a test-construction lemma that has been established in the literature; the idea comes from Genin \& Kelly (2017, Theorem 4.2, the 1-to-3 Side), but the exact lemma used here is provided by Boeken, Skapinakis, Genin \& Mooij (2026, Theorem 3, the 2-to-1 Side):

\begin{lemma}[Boeken, Skapinakis, Genin \& Mooij (2026)]
In any hypothesis testing problem $H_0$ versus $H_1$, if each hypothesis $H_i$ with $i \in \{0, 1\}$ is a countable union of closed sets in $H_0 \cup H_1$, there exists a finite-precision test that is everywhere (i.e., pointwise) consistent. 
\end{lemma}

The rest of this section is devoted to proving Theorem 2.

\begin{lemma}
Part $1 \To 2$ of Theorem 2 holds. That is, in any hypothesis testing problem, if the two hypotheses are nowhere intertwined, there exists a finite-precision test that is almost everywhere consistent in each of the two hypotheses.
\end{lemma}

\begin{proof}
{\em Part} $1 \Rightarrow 2$: The goal is to construct disjoint sets $A_0$ and $A_1$, both being a countable union of closed sets, such that Lemma 2 is applicable. There is no need for $A_i$ to be identical to $H_i$; it suffices that, for each $i \in \{0, 1\}$, the intersection $A_i \cap H_i$ covers $H_i$ almost everywhere---that is, $H_i \setminus A_i$ is nowhere dense in $H_i$ (as a topological space in its own right). Throughout, $\text{int}(\,\cdot\,)$ denotes the interior taken in the ambient space $H_0 \cup H_1$.

Since ${\cal X}$ is Polish, the space of probability measures under the weak topology is Polish, and hence so is its subspace $H_0 \cup H_1$; in particular $H_0 \cup H_1$ is second countable. Fix a countable base $O^{(1)}, O^{(2)}, \ldots$ for $H_0 \cup H_1$. (It is essential that the $O^{(i)}$ form a base, not merely a cover: the argument below places each point of $H_i$ that needs to be captured inside an \emph{arbitrarily small} basic set of the appropriate type.) Each $O^{(i)}$ falls under one of the following cases:
	\op 
	\im Case $(a)$: Both $H_0 \cap O^{(i)}$ and $H_1 \cap O^{(i)}$ have a nonempty open subset in $O^{(i)}$. 
		\op
		\im Let $A_0^{(i)}$ be a nonempty subset of $H_0 \cap O^{(i)}$ that is open in $H_0 \cup H_1$.
		\im Let $A_1^{(i)}$ be a nonempty subset of $H_1 \cap O^{(i)}$ that is open in $H_0 \cup H_1$.
		\ed 
	
	\im Case $(b)$: $H_0 \cap O^{(i)}$ has no nonempty open subset in $O^{(i)}$, and $H_1 \cap O^{(i)}$ has a nonempty open subset in $O^{(i)}$.
		\op
		\im Then $H_1 \cap O^{(i)}$ is dense in $O^{(i)}$. So $H_0 \cap O^{(i)}$ is nowhere dense in $O^{(i)}$; for otherwise its closure would contain a nonempty open subset $W$ of $O^{(i)}$, on which $H_0$ is dense while $H_1$ remains dense, making $H_0$ and $H_1$ intertwined in $W$---a contradiction.
		\im Let $A_0^{(i)} = O^{(i)} \setminus \text{int}(H_1)$. 
		\im Let $A_1^{(i)}$ be a nonempty subset of $H_1 \cap O^{(i)}$ that is open in $H_0 \cup H_1$.
		\ed 
		
	\im Case $(c)$: This is the dual of the preceding case; that is, $H_0 \cap O^{(i)}$ has a nonempty open subset in $O^{(i)}$, and $H_1 \cap O^{(i)}$ has no nonempty open subset in $O^{(i)}$.
		\op
		\im Then, for the same reason, $H_1 \cap O^{(i)}$ is nowhere dense in $O^{(i)}$.
		\im Let $A_0^{(i)}$ be a nonempty subset of $H_0 \cap O^{(i)}$ that is open in $H_0 \cup H_1$.
		\im Let $A_1^{(i)} = O^{(i)} \setminus \text{int}(H_0)$.
		\ed 
	\ed 
These three cases are exhaustive. The only remaining possibility is that neither $H_0 \cap O^{(i)}$ nor $H_1 \cap O^{(i)}$ has a nonempty open subset in $O^{(i)}$; but since these two sets partition $O^{(i)}$, each having empty interior in $O^{(i)}$ means each is the complement of a set with empty interior, hence each is dense in $O^{(i)}$---so $H_0$ and $H_1$ would be intertwined in $O^{(i)}$, contrary to hypothesis. (Note that this uses only nowhere-intertwinedness, not the Baire property.) Now define
	\op 
	\im $A_0 = \bigcup_{i \in \mathbb{N}} A_0^{(i)}$,
	\im $A_1 = \bigcup_{i \in \mathbb{N}} A_1^{(i)}$.
	\ed 

\medskip
\noindent\textit{Disjointness of $A_0$ and $A_1$.} We first record that any $O^{(i)}$ in case $(b)$ is disjoint from any $O^{(j)}$ in case $(c)$. Otherwise $W := O^{(i)} \cap O^{(j)}$ is a nonempty open set. As an open subset of $O^{(i)}$, it inherits that $H_0$ is nowhere dense in $W$, so $H_1$ is dense in $W$; as an open subset of $O^{(j)}$, it inherits that $H_1$ is nowhere dense in $W$, so $H_0$ is dense in $W$. Then $H_0$ and $H_1$ are both dense in $W$---intertwined---contradicting the hypothesis. (Again this uses only nowhere-intertwinedness.)

It suffices to check $A_0^{(i)} \cap A_1^{(j)} = \varnothing$ for all $i, j$. Each $A_0^{(i)}$ is of one of two forms: a nonempty open subset of $H_0$ (cases $(a)$, $(c)$), which therefore lies in $\text{int}(H_0)$; or $O^{(i)} \setminus \text{int}(H_1)$ (case $(b)$), which is disjoint from $\text{int}(H_1)$. Dually, each $A_1^{(j)}$ is either an open subset of $H_1$ (cases $(a)$, $(b)$), lying in $\text{int}(H_1)$, or $O^{(j)} \setminus \text{int}(H_0)$ (case $(c)$), disjoint from $\text{int}(H_0)$. Since $\text{int}(H_0) \cap \text{int}(H_1) = \varnothing$, the only combination not immediately yielding disjointness is $A_0^{(i)} = O^{(i)} \setminus \text{int}(H_1)$ with $A_1^{(j)} = O^{(j)} \setminus \text{int}(H_0)$; but then $O^{(i)}$ is in case $(b)$ and $O^{(j)}$ in case $(c)$, so $O^{(i)} \cap O^{(j)} = \varnothing$ by the preceding paragraph, whence $A_0^{(i)} \subseteq O^{(i)}$ and $A_1^{(j)} \subseteq O^{(j)}$ are disjoint. Therefore $A_0 \cap A_1 = \varnothing$.

\medskip
\noindent\textit{Covering.} We claim that, for each $i \in \{0,1\}$, the set $H_i \setminus A_i$ is nowhere dense in $H_i$. We use the following elementary characterization: for a subset $S$ of a topological space $Y$, the complement $Y \setminus S$ is nowhere dense in $Y$ if and only if every nonempty open subset of $Y$ contains a nonempty open subset that is included in $S$ (equivalently, $\text{int}_Y S$ is dense in $Y$). This is exactly the intuition guiding the construction.

Take $i = 0$; the case $i = 1$ is symmetric. Let $Y = H_0$ and $S = A_0 \cap H_0$, and let $W$ be a nonempty relatively open subset of $H_0$, say $W = O \cap H_0$ with $O$ open in $H_0 \cup H_1$ and $O \cap H_0 \neq \varnothing$. Pick a point $p \in O \cap H_0$; since the $O^{(k)}$ form a base, there is one with $p \in O^{(k)} \subseteq O$, and then $O^{(k)} \cap H_0 \neq \varnothing$. According to the case of $O^{(k)}$:
	\op
	\im If $O^{(k)}$ is in case $(a)$ or $(c)$, then $A_0^{(k)}$ is a nonempty open subset of $H_0 \cap O^{(k)}$. It is relatively open in $H_0$, included in $A_0 \cap H_0$, and contained in $O \cap H_0 = W$.
	\im If $O^{(k)}$ is in case $(b)$, then $O^{(k)} \cap H_0 \subseteq O^{(k)} \setminus \text{int}(H_1) = A_0^{(k)}$, because $H_0$ is disjoint from $\text{int}(H_1)$. Thus $O^{(k)} \cap H_0$ is a nonempty relatively open subset of $H_0$, included in $A_0 \cap H_0$, and contained in $W$.
	\ed
Either way, $W$ contains a nonempty relatively open subset included in $A_0 \cap H_0$. By the characterization, $H_0 \setminus A_0$ is nowhere dense in $H_0$.

\medskip
\noindent\textit{The sets are $F_\sigma$, and Lemma 2 applies.} Each $A_0^{(i)}$ is either open (cases $(a)$, $(c)$) or of the form $O^{(i)} \cap \bigl(H_0 \cup H_1 \setminus \text{int}(H_1)\bigr)$, the intersection of an open set with a closed set (case $(b)$). In a metrizable space every open set is a countable union of closed sets, and intersecting each such closed set with a fixed closed set keeps it closed; hence each $A_0^{(i)}$ is $F_\sigma$, and the countable union $A_0$ is $F_\sigma$. Likewise $A_1$ is $F_\sigma$. Each $A_i$ is thus $F_\sigma$ in $H_0 \cup H_1$, and therefore $F_\sigma$ in the subspace $A_0 \cup A_1$ as well.

By Lemma 2 applied to the hypothesis testing problem $A_0$ versus $A_1$, there is a test $\phii$ that is everywhere consistent on $A_0 \cup A_1$: for each $\PP \in A_i$, $\PP^{\,n}\{\phii_n = i\} \to 1$. This $\phii$ has the required properties for the original problem $H_0$ versus $H_1$. Indeed, fix $i \in \{0,1\}$. At every $\PP \in H_i \cap A_i$, the test outputs $i$ in the limit, which is correct. Every $\PP \in H_i$ at which $\phii$ may fail lies outside $A_i$: if $\PP \in A_{1-i}$ then $\phii$ misclassifies it, but $A_i \cap A_{1-i} = \varnothing$ puts $\PP$ in $H_i \setminus A_i$; and if $\PP \notin A_0 \cup A_1$, then again $\PP \in H_i \setminus A_i$. Since $H_i \setminus A_i$ is nowhere dense in $H_i$, the set of measures in $H_i$ at which $\phii$ is consistent equals $H_i$ minus a subset nowhere dense in $H_i$. As this holds for each $i$, the test $\phii$ achieves almost everywhere consistency in each hypothesis, establishing condition~2.

\end{proof}

\begin{proof}[Proof of Theorem 2] The proof will proceed in this order: $1 \Rightarrow 2 \Rightarrow 3 \Rightarrow 4 \Rightarrow 1$. 

Part $1 \Rightarrow 2$ has been established in Lemma 3. Part $2 \Rightarrow 3 \Rightarrow 4$ involves only routine verifications; the one point worth recording is that in $2 \Rightarrow 3$ a set nowhere dense in $H_i$ is also nowhere dense in the ambient space $H_0 \cup H_1$, so the two per-hypothesis exceptional sets together form a nowhere dense subset of $H_0 \cup H_1$, which yields almost everywhere consistency overall. 

For part $4 \Rightarrow 1$, suppose that condition 4 holds: there exists a finite-precision test $\phii$ for the problem $H_0$ versus $H_1$ that is densely consistent in each hypothesis and comeagerly consistent. Suppose for {\em reductio} that condition 1 does not hold; then, in the topological space $H_0 \cup H_1$, there is a nonempty open set $O$ on which both $H_0$ and $H_1$ are dense. Consider the ``smaller'' hypothesis testing problem $H'_0$ versus $H'_1$, where $H'_i =_\textrm{df} H_i \cap O$; its ambient space $H'_0 \cup H'_1$ is $O$ itself. Since $H_i$ is dense in $O$, each $H'_i = H_i \cap O$ is dense in $O$, so the smaller problem satisfies the density hypothesis of Theorem 1. Moreover, $\phii$ retains, for the smaller problem, the properties needed: (i) it is still finite-precision, since its acceptance regions $\{\phii_n = i\}$ are unchanged, and still densely consistent in each hypothesis, since if $\phii$ is consistent on a dense subset $D_i \subseteq H_i$ then $D_i \cap O$ is dense in $H_i \cap O = H'_i$ $($a dense set meets every nonempty relatively open subset$)$ and $\phii$ is consistent there; and (ii) it is still comeagerly consistent, since if $\phii$ is consistent on a comeager subset $C \subseteq H_0 \cup H_1$ then $C \cap O$ is comeager in $O$ $($meager sets restrict to meager sets in an open subspace$)$ and $\phii$ is consistent there. Applying Theorem 1 to (i), we obtain (iii): $\phii$ is not nonmeagerly consistent in the smaller problem---that is, its consistent set within $O$ is meager in $O$. But by (ii) that same consistent set is comeager in $O$. Hence $O$ is a union of two subsets that are meager in $O$, so $O$ is meager in itself, and therefore meager in $H_0 \cup H_1$. Thus $H_0 \cup H_1$ has a nonempty open subset that is meager---contradicting the assumption, in the statement of Theorem 2, that $H_0 \cup H_1$ is a Baire space, in which no nonempty open set is meager. This establishes $4 \Rightarrow 1$. 
\end{proof}

\section{More on ``Nowhere Intertwined''}

Nowhere-intertwinedness admits a simple structural characterization:

\begin{proposition}
For any hypothesis testing problem $H_0$ versus $H_1$, the following are equivalent:
\ope
\im The two hypotheses $H_0$ and $H_1$ are nowhere intertwined.
\im $\mathrm{int}(H_0) \cup \mathrm{int}(H_1)$ is dense in $H_0 \cup H_1$.
\im Each hypothesis $H_i$ is the union of an open set and a nowhere dense set $($both taken in $H_0 \cup H_1)$.
\ede
In particular, under nowhere-intertwinedness each $H_i$ is either open, nowhere dense, or the union of a nonempty open set with a nonempty nowhere dense set.
\end{proposition}

\begin{proof}
Throughout, $\text{int}(\,\cdot\,)$, ``dense'', and ``nowhere dense'' are taken in the ambient space $H_0 \cup H_1$.

\smallskip
\noindent$(1 \Leftrightarrow 2)$. Since $H_0$ and $H_1$ partition any open set $O$, the trace $H_0 \cap O$ has empty interior in $O$ if and only if its complement $H_1 \cap O$ is dense in $O$. Hence both hypotheses are dense in $O$ if and only if both have empty interior in $O$, i.e.\ if and only if $O$ is disjoint from $\text{int}(H_0) \cup \text{int}(H_1)$. Therefore the hypotheses are intertwined in some nonempty open set if and only if $\text{int}(H_0) \cup \text{int}(H_1)$ fails to meet some nonempty open set---that is, fails to be dense. Negating both sides gives $(1) \Leftrightarrow (2)$.

\smallskip
\noindent$(2 \Rightarrow 3)$. Suppose $\text{int}(H_0) \cup \text{int}(H_1)$ is dense. This set is open, so its complement
\[
	R := (H_0 \cup H_1) \setminus \bigl(\text{int}(H_0) \cup \text{int}(H_1)\bigr)
\]
is closed with empty interior, hence nowhere dense. For each $i$, a point of $H_i \setminus \text{int}(H_i)$ lies in neither $\text{int}(H_0)$ nor $\text{int}(H_1)$ (it is not interior to $H_i$, and it is not in $H_{1-i}$ at all), so $H_i \setminus \text{int}(H_i) \subseteq R$ and is therefore nowhere dense. Thus $H_i = \text{int}(H_i) \cup \bigl(H_i \setminus \text{int}(H_i)\bigr)$ exhibits $H_i$ as an open set together with a nowhere dense set.

\smallskip
\noindent$(3 \Rightarrow 2)$. Suppose $H_i = U_i \cup N_i$ with $U_i$ open and $N_i$ nowhere dense, for each $i$. Then $U_i \subseteq \text{int}(H_i)$, and
\[
	(H_0 \cup H_1) \setminus \bigl(\text{int}(H_0) \cup \text{int}(H_1)\bigr) \subseteq (H_0 \cup H_1) \setminus (U_0 \cup U_1) \subseteq N_0 \cup N_1,
\]
a nowhere dense set. A set whose complement is contained in a nowhere dense set is dense, so $\text{int}(H_0) \cup \text{int}(H_1)$ is dense.
\end{proof}

\section{Another Characterization Result}

Once we have considered ``nowhere intertwined'', it seems natural to define ``everywhere intertwined'': in a hypothesis testing problem $H_0$ versus $H_1$, the two hypotheses are called {\em everywhere intertwined} iff they are intertwined in every nonempty open subset of the ambient space $H_0 \cup H_1$. In fact, this condition holds exactly when both hypotheses are dense in $H_0 \cup H_1$. This returns us to the main impossibility result---whose converse, it turns out, also holds under the Baire assumption---and yields a second characterization:

\begin{theorem}
For any hypothesis testing problem $H_0$ versus $H_1$ whose ambient space $H_0 \cup H_1$ is a Baire space, the following two conditions are equivalent:
\ope  
\im The two hypotheses are everywhere intertwined.
\im There exists no finite-precision test that achieves the $($extremely low$)$ standard of $($i$)$ dense consistency in each hypothesis and $($ii$)$ nonmeager consistency.
\ede  
\end{theorem}

\begin{proof} $1 \Rightarrow 2$: This is just the impossibility result stated above, i.e., Theorem 1. 

$2 \Rightarrow 1$: We argue by contraposition: assuming condition 1 fails, we produce a finite-precision test that is densely consistent in each hypothesis and nonmeagerly consistent, so that condition 2 fails as well.

Since condition 1 fails, the two hypotheses are not intertwined in some nonempty open set $O \subseteq H_0 \cup H_1$; that is, they are not both dense in $O$. Say $H_1$ is not dense in $O$ (the other case is symmetric). Then some nonempty open $O' \subseteq O$ misses $H_1$, so $O' \subseteq H_0$; being open in $H_0 \cup H_1$, this $O'$ will supply the nonmeagerness.

Because $H_0 \cup H_1$ is Polish, each $H_i$ is separable; fix a countable dense subset $D_i \subseteq H_i$. Put
\[
	H'_0 := D_0 \cup O', \qquad H'_1 := D_1 .
\]
These are disjoint, since $H'_0 \subseteq H_0$, $H'_1 \subseteq H_1$, and $H_0 \cap H_1 = \varnothing$. Each is a countable union of closed sets in the subspace $H'_0 \cup H'_1$: the countable sets $D_0$ and $D_1$ are countable unions of (closed) singletons, and $O'$, being relatively open in $H'_0 \cup H'_1$, is $F_\sigma$ there because in a metrizable space every open set is a countable union of closed sets. Hence $H'_0$ and $H'_1$ meet the hypothesis of Lemma 2.

Applying Lemma 2 to the subproblem $H'_0$ versus $H'_1$ yields a finite-precision test $\phii$ that is everywhere consistent on $H'_0 \cup H'_1$: it outputs $0$ in the limit at each $\PP \in H'_0$ and $1$ in the limit at each $\PP \in H'_1$. Finite precision is a property of the acceptance regions $\{\phii_n = i\} \subseteq {\cal X}^n$ alone, so $\phii$ remains finite-precision when regarded as a test for the original problem $H_0$ versus $H_1$. We verify the two required properties there.

First, $\phii$ is densely consistent in each hypothesis: it outputs $0$ in the limit at every $\PP \in D_0$ and $1$ in the limit at every $\PP \in D_1$, and $D_0, D_1$ are dense in $H_0, H_1$ respectively. Second, $\phii$ is nonmeagerly consistent: at every $\PP \in O' \subseteq H_0$ it outputs $0$ in the limit, which is correct since $O' \subseteq H_0$, so its consistent set (in the original problem) contains the nonempty open set $O'$; and a nonempty open subset of the Baire space $H_0 \cup H_1$ is nonmeager. Thus $\phii$ witnesses the failure of condition 2, which completes the contraposition.
\end{proof}

\section*{References}

\begin{description}
\im Angrist, J. D., Imbens, G. W., \& Rubin, D. B. (1996) Identification of Causal Effects Using Instrumental Variables. {\em Journal of the American Statistical Association} 91(434): 444--455.

\im Banach, S. (1931) \"{U}ber die Baire'sche Kategorie gewisser Funktionenmengen. {\em Studia Mathematica} 3: 174--179.

\im Boeken, P., Skapinakis, E., Genin, K., \& Mooij, J. M. (2026) Topological Criteria for Hypothesis Testing with Finite-Precision Measurements. arXiv preprint arXiv:2601.13946.

\im Boeken, P., Forr\'{e}, P., \& Mooij, J. M. (2026) Are Bayesian Networks Typically Faithful? arXiv preprint arXiv:2410.16004.

\im Genin, K., \& Kelly, K. T. (2017) The Topology of Statistical Verifiability. In {\em Proceedings of the 16th Conference on Theoretical Aspects of Rationality and Knowledge (TARK 2017)} (J. Lang, ed.), Electronic Proceedings in Theoretical Computer Science 251: 236--250.

\im Kelly, K. T. (1996) {\em The Logic of Reliable Inquiry}. New York: Oxford University Press.

\im Lin, H., \& Zhang, J. (2020) On Learning Causal Structures from Non-Experimental Data without Any Faithfulness Assumption. In {\em Proceedings of the 31st International Conference on Algorithmic Learning Theory} (A. Kontorovich \& G. Neu, eds.), Proceedings of Machine Learning Research 117: 554--582. PMLR.

\im Neyman, J., \& Pearson, E. S. (1933) On the Problem of the Most Efficient Tests of Statistical Hypotheses. {\em Philosophical Transactions of the Royal Society of London, Series A} 231: 289--337.

\im Neyman, J., \& Pearson, E. S. (1936) Contributions to the Theory of Testing Statistical Hypotheses. {\em Statistical Research Memoirs} 1: 1--37.

\im Spirtes, P., Glymour, C., \& Scheines, R. (2000) {\em Causation, Prediction, and Search}, 2nd edition. Cambridge, MA: MIT Press.
\end{description}

\end{document}